\documentclass[10pt]{amsart}

\usepackage{amsmath}
\usepackage{hyperref}
\usepackage{graphicx}
\usepackage{epsfig}
\usepackage{latexsym}
\usepackage{bm}
\usepackage{amssymb}

\newtheorem{theorem}{Theorem}[section]

\newtheorem{lemma}[theorem]{Lemma}

\theoremstyle{remark}
\newtheorem{remark}[theorem]{Remark}

\numberwithin{equation}{section}

\usepackage{hyperref}

\begin{document}

\title[On node distributions]{On node distributions for interpolation and spectral methods}

\author{N. S. Hoang}

\address{Mathematics Department, University of Oklahoma,
Norman, OK 73019, USA}

\email{nhoang@math.ou.edu}


\date{}


\begin{abstract}
A scaled Chebyshev node distribution is studied in this paper. It is proved that the node distribution is optimal for interpolation in $C_M^{s+1}[-1,1]$, the set of $(s+1)$-time differentiable functions whose $(s+1)$-th derivatives are bounded by a constant $M>0$. Node distributions for computing spectral differentiation matrices are proposed and studied. 
Numerical experiments show that the proposed node distributions yield results with higher accuracy than the most commonly used Chebyshev-Gauss-Lobatto node distribution. 
\end{abstract}
\maketitle

\section{Introduction}
Choosing nodes is important in interpolating a function and solving differential or integral equations by pseudospectral methods. 
Given a sufficiently smooth function, if nodes are not suitably chosen, then the interpolation polynomials do not converge to the function as the number of nodes tends to infinity. A well-known example is the Runge's phenomenon. In particular, if one uses equi-spaced nodes to interpolate the Runge's function $f(x)=\frac{1}{1+25x^2}$ over the interval $[-1,1]$, then the errors of Lagrange polynomial interpolation blow up to infinity as the number of nodes increases (see, e.g., \cite{FB}). 

Let $f$ be a continuous function on $[-1,1]$, let $\bm{c} := (c_i)_{i=0}^{s}$, $c_i\in [-1,1]$, and let $L_{\bm{c}}(f)(x)$ 
be the Lagrange interpolation polynomial of $f$ over the nodes $(c_i)_{i=0}^{s}$. 
It is well-known from interpolation theory that 
\begin{equation}
\max_{x\in [-1,1]}|L_{\bm{c}}(f)(x)-f(x)| \le (1+\Lambda(\bm{c}))\max_{x\in [-1,1]}|P^*(x) - f(x)|,
\end{equation}
where $P^*(x)$ is the best polynomial approximation of degree $s$ and $\Lambda(\bm{c})$ is the Lebesgue constant corresponding to the node distribution $\bm{c}=(c_i)_{i=0}^{s}$. The Lebesgue constant $\Lambda(\bm{c})$ indicates how far the Lagrange interpolation polynomial $L_{\bm{c}}(x)$ is from the best polynomial approximation of degree $s$. Lebesgue constants have been studied extensively in the literature (see, e.g, \cite{Brutman}, \cite{FB}, \cite{Henry}, \cite{Rack}, \cite{Smith}, \cite{Trefethen}, \cite{Hes98}, and references therein). 
It is of interest to find a node distribution for which the Lebesgue constant is minimal among all node distributions with the same number of nodes. This node distribution if existing is called an optimal node distribution. It is known that for a given number of nodes, the optimal node distribution may not be unique. If one wants these nodes to include boundary points, then such optimal node distribution is unique (cf. \cite{Henry}). 
However, finding these node distributions is not an easy task. In practice, one often uses Chebyshev-Gauss-Lobatto nodes for interpolation and pseudospectral methods. These nodes are extrema of Chebyshev polynomials of the first kind over $[-1,1]$. 

The most commonly used node distribution is 
Gauss-Chebyshev-Lobatto points. These points are extrema of Chebyshev polynomial $T_{s}$ over $[-1,1]$, i.e., 
\begin{equation}
\label{eq5}
c_i = \cos(\frac{i\pi}{s}),\qquad i=0,...,s. 
\end{equation}
This node distribution is also referred to as Chebyshev points. 
In \cite{Hes98} the Lebesgue constant of this node distribution was studied. It was proved that the Lebesgue constant for Chebyshev-Gauss-Lobatto nodes in \eqref{eq5} satisfies the estimate (see, e.g., \cite{FB}, \cite{Hes98})
\begin{equation}
\label{eq42}
\Lambda^{CGL}(n) = \frac{2}{\pi}\bigg(\ln n + \gamma + \ln\frac{8}{\pi}\bigg) +  O\big(\frac{1}{n^2}\big),
\end{equation}
where $\gamma=0.577215$ is the Euler constant and $n$ is the number of nodes.  
 
Although Chebyshev-Gauss-Lobatto node distribution works well in practice, it is not optimal in the sense that the Lebesgue constant for this node distribution is minimal among Lebesgue constants based on node distributions of the same number of nodes. 
It is well-known that for each function $f$ there is an optimal node distribution for interpolating the function. 
This optimal node distribution varies from functions to functions. When $f$ is known, there are algorithms for finding an optimal node distribution for interpolating $f$. However, these algorithms are not efficient in practice. In many cases, these algorithms are not applicable since the function to be interpolated is not known. This is the case when $f$ is a solution to a differential or an integral equation. 

It was proved that the optimal Lebesgue constant satisfies the following estimate
(see, e.g., \cite{Vertesi})
\begin{equation}
\label{eq43}
\Lambda^{min}(n) = \frac{2}{\pi}\bigg(\ln n + \gamma + \ln\frac{4}{\pi}\bigg) + O(\frac{1}{(\ln n)^{1/3}}). 
\end{equation}
From equations \eqref{eq42} and \eqref{eq43} one can see that the Lebesgue constant of Chebyshev-Gauss-Lobatto nodes is very close to the optimal one.

In \cite{Gunttner} the Lebesgue contant for a scaled Chebyshev node distribution was studied. These nodes are obtained by scaling zeros of the Chebyshev polynomial $T_{s+1}(x)$. 
In particular, the scaled Chebyshev nodes $(c_i)_{i=1}^s$ in \cite{Gunttner} are defined as follows
\begin{equation}
\label{eq6}
c_i = \cos\bigg(\frac{2i+1}{2(s+1)}\pi\bigg)\bigg[\cos\bigg(\frac{1}{2(s+1)}\pi\bigg)\bigg]^{-1},\qquad i=0,...,s. 
\end{equation}
The Lebesgue constant of the scaled Chebyshev node distribution satisfies the following estimate (see, e.g., \cite{Gunttner}, \cite{Smith})
\begin{equation}
\Lambda^{sC}(n) = \frac{2}{\pi}\bigg(\ln n + \gamma + \ln\frac{8}{\pi} - \frac{2}{3}\bigg) +  O\big(\frac{1}{\ln n}\big).
\end{equation}  
Note that $\ln\frac{8}{\pi} - \frac{2}{3} \approx \ln\frac{4}{\pi}+0.0265$ and $\ln\frac{8}{\pi} \approx \ln\frac{4}{\pi} + 0.24$. 
Thus, for ``large'' $n$, the Lebesgue constants of the scaled Chebyshev points are closer to the optimal Lebesgue contants compared to the Lebesgue constants of Chebyshev-Gauss-Lobatto points. The scaled Chebyshev nodes are often mentioned as the optimal choice in practice for interpolation (cf. \cite{Henry}). However, to the author's knowledge, there is no justification for the optimality of this choice in any sense. 

In practice one often uses Chebyshev-Gauss-Lobatto nodes and scaled Chebyshev nodes for interpolation and 
pseudospectral methods. 

In this paper, we study node distributions for interpolation and pseudospectral methods over the class of functions 
$$
C_M^{s+1}[-1,1]:=\{f\in C^{s+1}[-1,1]: \max_{x\in[-1,1]}|f^{(s+1)}(x)|\le M\}.
$$
It turns out that the scaled Chebyshev nodes are optimal for interpolation over $C_M^{s+1}[-1,1]$. We also construct node distributions for computing differentiation matrices over $C_M^{s+1}[-1,1]$. Numerical experiments with the new node distributions in Section \ref{sec4} (see below) showed that these nodes yield better results than Chebyshev-Gauss-Lobatto points do. 

The paper is organized as follows. 
In Section \ref{sec2} we study node distributions for interpolation. 
We prove that the scaled Chebyshev nodes are ``optimal'' for interpolation over $C_M^{s+1}[-1,1]$.
In Section \ref{sec3} node distributions for calculating differentiation matrices are proposed and justified. 
In Section \ref{sec4} numerical experiments are carried out with the new node distributions.

\section{Interpolation}
\label{sec2}

Let $L_{\bm{c}}(f)$ denote the Lagrange interpolation polynomial of a sufficiently smooth function $f$ over the nodes $\bm{c}=(c_i)_{i=0}^s$, $c_i\in [-1,1]$. 
The error of Lagrange interpolation is given by the formula (see, e.g., \cite{Kress})
\begin{equation}
\label{eq1}
L_{\bm{c}}(f)(x)-f(x) = \frac{f^{(s+1)}(\xi(x))}{(s+1)!}\prod_{i=0}^s (x-c_i),\qquad \xi(x)\in[-1,1].
\end{equation}
We are interested in finding a node distribution $\bm{c}$ so that the interpolation error $\|L_{\bm{c}}(f)-f\|_\infty$ is as small as possible. Here, $\|g\|_\infty$ denotes the sup-norm of $g$ over the interval $[-1,1]$, i.e., 
$\|g\|_\infty:=\sup_{x\in[-1,1]}|g(x)|$. 
Note that the element $\xi(x)$ in \eqref{eq1} depends on $x$ and $(c_i)_{i=0}^s$ in a nontrivial manner. Therefore, to minimize $\|L_{\bm{c}}(f)-f\|_\infty$ one often tries to find a distribution of $(c_i)_{i=0}^s$, $c_i\in [-1,1]$, so that 
\begin{equation}
\label{eq2}
\max_{-1\le x\le 1}\bigg|\prod_{i=0}^s (x-c_i)\bigg|\to \min.
\end{equation}
It is well-known that the zeros of $T_{s+1}(x)$, 
the Chebyshev polynomial of order $s+1$ of the first kind over $[-1,1]$, are the solution to \eqref{eq2}. These zeros are given by the formula
\begin{equation}
\label{eq3}
c_i = \cos\bigg(\frac{2i+1}{2(s+1)}\pi\bigg),\qquad i=0,...,s. 
\end{equation}

In practice one often wants to have boundary points as interpolation nodes, i.e., $c_0=-1$ and $c_s = 1$. Let  
\begin{equation}
\mathcal{C}:=\{\bm{c}=(c_i)_{i=0}^s: -1=c_0<c_1<...<c_{s-1}<c_s = 1\}. 
\end{equation}
The following question arises: for which set of points $(c_i)_{i=0}^{s}\in \mathcal{C}$, we have
\begin{equation}
\label{eq4}
\max_{-1\le x\le 1}\bigg|\prod_{i=0}^{s} (x-c_i)\bigg|\to \min_{\bm{c}\in \mathcal{C}}\, ?
\end{equation}
The answer is given in the following result:

\begin{theorem}
\label{theorem1}
Let $(\bar{c}_i)_{i=0}^s\in \mathcal{C}$ be a solution to \eqref{eq4}, i.e., 
\begin{equation}
\label{eq8}
\max_{-1\le x\le 1}\bigg|\prod_{i=0}^{s}(x-\bar{c}_i)\bigg| = \min_{\bm{c}\in \mathcal{C}} \max_{-1\le x\le 1}\bigg|\prod_{i=0}^{s}(x-c_i)\bigg|.
\end{equation}
Then this solution is uniquely and $(\bar{c}_i)_{i=0}^s$ are determined by
\begin{equation}
\label{eq9}
\bar{c}_i = \cos\bigg(\frac{2i+1}{2(s+1)}\pi\bigg)\bigg[\cos\bigg(\frac{\pi}{2(s+1)}\bigg)\bigg]^{-1},\qquad i=0,...,s.
\end{equation}
\end{theorem}

\begin{proof}
Let 
$$
P(x) := \prod_{i=0}^s (x-\bar{c}_i),
$$
where $\bar{c}_i$, $i=0,...,s$, are defined by \eqref{eq9}. 
Then
\begin{equation}
\begin{split}
P\bigg(\frac{x}{\cos(\frac{\pi}{2(s+1)})}\bigg) &= \bigg[\frac{1}{\cos(\frac{\pi}{2(s+1)})}\bigg]^{s+1}\prod_{i=0}^s \bigg(x-\cos\big(\frac{2i+1}{2(s+1)}\pi\big)\bigg) \\
&= \bigg[\frac{1}{\cos(\frac{\pi}{2(s+1)})}\bigg]^{s+1}T_{s+1}(x),
\end{split}
\end{equation}
where $T_{s+1}(x)$ is the Chebyshev polynomial of the first kind over $[-1,1]$ of degree $s+1$. Therefore,
\begin{equation}
\label{eq7}
P(x) = \bigg[\frac{1}{\cos(\frac{\pi}{2(s+1)})}\bigg]^{s+1}T_{s+1}\bigg(x\cos(\frac{\pi}{2(s+1)})\bigg).
\end{equation}
Note that $\big(\cos(\frac{i\pi}{s+1})\big)_{i=1}^{s}$ are all critical points of the Chebyshev polynomial $T_{s+1}(x)$ and $|T_{s+1}(x)|\le 1$, $\forall x\in [-1,1]$. This and equation \eqref{eq7} imply that all critical points of $P(x)$ are
\begin{equation}
\label{eq13}
d_i = \frac{\cos(\frac{i\pi}{s+1})}{\cos(\frac{\pi}{2(s+1)})},\qquad i=1,...,s,
\end{equation}
and we have
\begin{equation}
\label{au4eq1}
P(d_i) = \bigg[\frac{1}{\cos(\frac{\pi}{2(s+1)})}\bigg]^{s+1} (-1)^i,\qquad i=1,...,s.
\end{equation}
Therefore, 
\begin{equation}
\label{eq10}
\min_{(c_i)_{i=0}^s\in\mathcal{C}}\max_{-1\le x\le 1}\bigg|\prod_{i=0}^{s}(x-c_i)\bigg| \le
\max_{x\in[-1,1]}|P(x)| = \bigg[\frac{1}{\cos(\frac{\pi}{2(s+1)})}\bigg]^{s+1}. 
\end{equation}

Let $(\tilde{c}_i)_{i=0}^s$ be a solution to \eqref{eq4}. Let us prove that $\tilde{c}_i=\bar{c}_i$ where $\bar{c}_i$, $i=0,..,s$, are defined by \eqref{eq9}. 
Let 
\begin{equation}
\label{eq14}
Q(x) := \prod_{i=0}^s(x-\tilde{c}_i),\qquad (\tilde{c}_i)_{i=0}^s \in \mathcal{C},
\end{equation}
and 
\begin{equation}
\label{eq11}
R(x) := Q(x) - P(x).
\end{equation}
Since $P(x)$ and $Q(x)$ are monic polynomials of degree $s+1$, one concludes from 
\eqref{eq11} that $R(x)$ is a polynomial of degree at most $s$. 

Since $(\tilde{c}_i)_{i=0}^s$ is a solution to \eqref{eq4} and \eqref{eq10} holds, one gets
\begin{equation}
\label{eq12}
|Q(x)| \le  \bigg[\frac{1}{\cos(\frac{\pi}{2(s+1)})}\bigg]^{s+1} ,\qquad \forall x\in [-1,1].
\end{equation}
From \eqref{au4eq1}, \eqref{eq11}, and \eqref{eq12}, one obtains
\begin{equation}
R(d_i)(-1)^i \ge 0,\qquad i=1,...,s.
\end{equation}
Thus, the polynomial $R(x)$ has at least $s-1$ zeros on the interval $[-d_1,d_{s}]\subset (-1,1)$. Since 
$\bar{c}_0=\tilde{c}_0=-1$ and $\bar{c}_s=\tilde{c}_s=1$, it is clear that 
$-1$ and $1$ are zeros of $Q(x)$ and $P(x)$. Thus, $-1$ and $1$ are also zeros of $R(x)$. Therefore, $R(x)$ has a total of $s+1$ zeros on the interval $[-1,1]$. This and the fact that $R(x)$ is a polynomial of degree at most $s$ imply that $R(x)=0$. Thus, $Q(x)\equiv P(x)$. Therefore, $\tilde{c}_i=\bar{c}_i$, $i=0,...,s$.

Theorem \ref{theorem1} is proved. 
\end{proof}

\begin{remark}{\rm
From the proof of Theorem \ref{theorem1}, one gets
\begin{equation}
\label{au5e2}
\min_{\bm{c}\in \mathcal{C}} \max_{-1\le x\le 1}\bigg|\prod_{i=0}^{s}(x-c_i)\bigg| = 
\max_{-1\le x\le 1}\bigg|\prod_{i=0}^{s}(x-\bar{c}_i)\bigg|= \bigg[\frac{1}{\cos(\frac{\pi}{2(s+1)})}\bigg]^{s+1}. 
\end{equation}
}
\end{remark}

Let 
\begin{equation}
\label{eq15}
C^{s+1}_M:=\big\{f\in C^{s+1}[-1,1]: \max_{x\in [-1,1]}|f^{(s+1)}(x)|\le M \big \},\quad M>0. 
\end{equation}
Let $L_{\bm{c}}(f)$ denote the Lagrange interpolation polynomial of $f$ over the nodes $\bm{c}:=(c_i)_{i=1}^s$. We are interested in solving the following problem
\begin{equation}
\label{eq40}
\min_{\bm{c}\in\mathcal{C}}\sup_{f\in C^{s+1}_M} \|f-L_{\bm{c}}(f)\|_\infty.
\end{equation}
Here $\|g\|_\infty:=\sup_{x\in [-1,1]}|g(x)|$. 

We have the following result:
\begin{theorem}
\label{theorem2}
Let $\bar{\bm{c}}:=(\bar{c}_i)_{i=0}^s$ where $(\bar{c}_i)_{i=0}^s$ are defined by \eqref{eq6}. Then $\bar{\bm{c}}$ is the solution to problem \eqref{eq40}. 
\end{theorem}

\begin{proof}
Let $\bm{c} = (c_i)_{i=0}^s\in \mathcal{C}$ be an arbitrary node distribution over $[-1,1]$. 
The error of Lagrange interpolation is given by the formula (see, e.g., \cite{Kress})
\begin{equation}
\label{eq28.0}
f(x) - L_{\bm{c}}(f)(x) = \frac{f^{(s+1)}(\xi(x))}{(s+1)!}\prod_{i=0}^s (x-c_i),\qquad \xi(x) \in [-1,1].
\end{equation}
From equations \eqref{eq28.0} and \eqref{au5e2} one gets
\begin{equation}
\label{eq17}
\begin{split}
\min_{\bm{c}\in\mathcal{C}}\sup_{f\in C^{s+1}_M} \|f-L_{\bm{c}}(f)\|_\infty 
&\le \sup_{f\in C^{s+1}_M} \|f-L_{\bar{\bm{c}}}(f)\|_\infty \\
&\le \frac{M}{(s+1)!}\max_{x\in [-1,1]}\bigg|\prod_{i=0}^s (x-\bar{c}_i)\bigg| \\
&= \frac{M}{(s+1)!}\bigg[\frac{1}{\cos(\frac{\pi}{2(s+1)})}\bigg]^{s+1}. 
\end{split}
\end{equation}

Let $P_{0}(x)$ be a polynomial of degree $s+1$ such that $P_0^{(s+1)}(x) \equiv M$. Using formula \eqref{eq28.0} for $f(x)=P_{0}(x)$, one gets
\begin{equation}
\label{jl31eq1}
P_0(x) - L_{\bm{c}}(P_0)(x) = \frac{P_0^{(s+1)}(\xi(x))}{(s+1)!}\prod_{i=0}^s (x-c_i) = \frac{M}{(s+1)!}\prod_{i=0}^s (x-c_i). 
\end{equation}
From equations \eqref{jl31eq1} and \eqref{au5e2} we have
\begin{equation}
\label{eq16}
\begin{split}
\min_{\bm{c}\in\mathcal{C}}\sup_{f\in C^{s+1}_M} \|f-L_{\bm{c}}(f)\|_\infty &\ge \min_{\bm{c}\in\mathcal{C}} \|P_0-L_{\bm{c}}(P_0)\|_\infty\\
& = \frac{M}{(s+1)!}\min_{\bm{c}\in\mathcal{C}}\max_{x\in[-1,1]}\bigg|\prod_{i=0}^s (x - c_i)\bigg| \\
& = \frac{M}{(s+1)!}\bigg[\frac{1}{\cos(\frac{\pi}{2(s+1)})}\bigg]^{s+1}.
\end{split}
\end{equation}

From equation \eqref{eq16} and \eqref{eq17}, we conclude that 
\begin{equation}
\min_{\bm{c}\in\mathcal{C}}\sup_{f\in C^{s+1}_M} \|f-L_{\bm{c}}(f)\|_\infty = \frac{M}{(s+1)!}\bigg[\frac{1}{\cos(\frac{\pi}{2(s+1)})}\bigg]^{s+1},
\end{equation}
and $\bar{\bm{c}}:=(\bar{c}_i)_{i=0}^s$, where $(\bar{c}_i)_{i=0}^s$ are defined by \eqref{eq6}, is the solution to \eqref{eq40}. 

Theorem \ref{theorem2} is proved. 
\end{proof}

\begin{remark}
Since the solution to \eqref{eq4} is unique, it follows from the proof of Theorem \ref{theorem2} that the solution to \eqref{eq40} is unique. 
Theorem \ref{theorem2} says that the node distribution from equation \eqref{eq6} is optimal in the sense of \eqref{eq40}. 
Namely, the node distribution defined by \eqref{eq6} is optimal for interpolation over the set of functions $C_M^{s+1}[-1,1]$.
\end{remark}

\section{Spectral differentiation matrices}
\label{sec3}

In many problems one is interested in finding the first derivative $f'$ of a function $f\in C^1[-1,1]$ based on  
values of $f$ at $(c_i)_{i=0}^s$, $c_i\in [-1,1]$. One of the approach is to use $(L_{\bm{c}}(f))'$ as an approximation to $f'$ where $L_{\bm{c}}(f)$ is the Lagrange interpolation polynomial of the function $f$ over the nodes $(c_i)_{i=0}^s$. Thus, the following problem arises
\begin{equation}
\label{eq28}
\min_{\bm{c}\in \mathcal{C}}\max_{x\in [-1,1]}|(L_{\bm{c}}(f))'(x) - f'(x)|,\qquad f\in C^1[-1,1]. 
\end{equation} 
Unfortunately, a solution $\bm{c}$ to \eqref{eq28} if existing is not 
independent of $f$, in general, and is not easy to find even when $f$ belongs to the  
class of functions $C^{s+1}_M[-1,1]$.

Let $f\in C^{s+1}[-1,1]$ and $\bm{c}=(c_i)_{i=0}^s$ be a node distribution over $[-1,1]$. 
Let $\lambda_k$ satisfy
\begin{equation}
\label{eq18.0}
f'(c_k) - (L_{\bm{c}}(f))'(c_k) = \lambda_k \frac{d}{dx}\prod_{i=0}^s (x-c_i)\bigg|_{x=c_k},\qquad 0\le k\le s. 
\end{equation} 
It is clear that $(c_i)_{i=0}^s$ are $s+1$ zeros of the function (cf. \eqref{eq1})
\begin{equation}
\label{eq29}
R_k(x) := f(x) - L_{\bm{c}}(f)(x) - \lambda_k \prod_{i=0}^s (x-c_i).
\end{equation}
According to Rolle's Theorem the function $R_k'(x)$ has at least $s$ zeros $(\eta_{ki})_{i=1}^s$ on the interval 
$[c_0,c_s]$ and $\eta_{ki}\not = c_j$. Therefore, $R_k'(x)$ has at least $s+1$ zeros on the interval $[-1,1]$ which are $(\eta_{ki})_{i=1}^s$ and $c_k$ (see \eqref{eq18.0}). Thus, by Rolle's Theorem, there exists $\zeta_k\in (-1,1)$ such that $R_k^{(s+1)}(\zeta_k) = 0$. This and \eqref{eq29} imply
\begin{equation}
0 = R_k^{(s+1)}(\zeta_k) = f^{(s+1)}(\zeta_k) - \lambda_k (s+1)!.
\end{equation}
Therefore, $\lambda_k = f^{(s+1)}(\zeta_k)/(s+1)!$ and we get from \eqref{eq18.0} the following relations
\begin{equation}
\label{eq18}
f'(c_k) - (L_{\bm{c}}(f))'(c_k) = \frac{f^{(s+1)}(\zeta_k)}{(s+1)!} \frac{d}{dx}\prod_{i=0}^s (x-c_i)\bigg|_{x=c_k},\qquad k=0,...,s. 
\end{equation}

Note that if $f\in C^{s+2}[-1,1]$, then 
equation \eqref{eq18} can also be obtained by differentiating equation \eqref{eq1} with respect to $x$ and assigning $x=c_k$. 


Fix $\bar{x}\in [-1,1]$
and $\bar{x}\not=c_i$, $i=0,...,s$. Let 
\begin{equation}
R(x) := f(x) - L_{\bm{c}}(f)(x) - \frac{f^{(s+1)}(\xi(\bar{x}))}{(s+1)!} \prod_{i=0}^s (x-c_i). 
\end{equation}
Then $R(x)$ has $s+2$ zeros which are $(c_i)_{i=0}^s$ and $\bar{x}$ (cf. \eqref{eq1}). Thus, by Rolle's Theorem, the function $R'(x)$ has at least 
$s+1$ zeros on $[-1,1]$. Let $(\eta_i)_{i=0}^s$ be zeros of $R'(x)$. Then one gets
\begin{equation}
\label{eq30}
f'(\eta_k) - (L_{\bm{c}}(f))'(\eta_k) = \frac{f^{(s+1)}(\xi_{\bar{x}})}{(s+1)!} \frac{d}{dx}\prod_{i=0}^s (x-c_i)\bigg|_{x=\eta_k},\qquad k=0,...,s. 
\end{equation}

From \eqref{eq18} and \eqref{eq30} one may ask whether or not there exists a constant $C>0$ such that
\begin{equation}
\label{jl30e1}
|f'(x) - (L_{\bm{c}}(f))'(x)| \le C  \bigg|\frac{d}{dx}\prod_{i=0}^s (x-c_i)\bigg|,\quad x\in [-1,1],\quad f\in C^{s+1}_M.
\end{equation}
Unfortunately, the answer to this question is negative. It is because zeros of the right side of \eqref{jl30e1} are, in general, not zeros of the left side of \eqref{jl30e1}. In particular, if $\xi$ is a zero of the right side of \eqref{jl30e1} but is not a zero of the left side of \eqref{jl30e1}, then equation \eqref{jl30e1} does not hold for any $C>0$ when  
$x=\xi$. 

To minimize the interpolation error $\|f' - (L_{\bm{c}}(f))'\|_\infty$, taking into account formulae \eqref{eq18} and \eqref{eq30}, 
we consider the following problem
\begin{equation}
\label{eq21}
\max_{-1\le x\le 1} \bigg|\frac{d}{dx}\prod_{i=0}^s (x-c_i)\bigg| \longrightarrow \min_{\bm{c}\in \mathcal{C}}. 
\end{equation}
From the theory of Chebyshev polynomials one concludes that the solution to problem \eqref{eq21} is a node distribution $(c_i)_{i=0}^s$ such that
\begin{equation}
\label{eq22}
\frac{d}{dx}\prod_{i=0}^s (x-c_i) = \frac{(s+1)T_{s}(x)}{2^{s-1}}.
\end{equation}
Thus, we want to find $(c_i)_{i=0}^s$ so that
\begin{equation}
\label{eq31}
\prod_{i=0}^s (x-c_i) = \int_0^x \frac{(s+1)T_{s}(\xi)}{2^{s-1}} d\xi + C,
\end{equation}
where $C$ is a suitable constant. 

To find $(c_i)_{i=0}^s$ satisfying \eqref{eq31} we need the following lemma:
\begin{lemma}
\label{lemma1}
Let $T_{s}$ be the Chebyshev polynomial of degree $s$ over the interval $[-1,1]$. Then
\begin{equation}
\label{eq23}
\int T_{s}(x)dx  = 
\frac{1}{2}\bigg(\frac{T_{s+1}(x)}{s+1} - \frac{T_{s-1}(x)}{s-1}\bigg) + C,\qquad C=const.
\end{equation}
\end{lemma}

\begin{proof}
One has
\begin{equation}
\label{eq24}
\begin{split}
\int T_{s}(x)dx &= \int \cos(s \arccos(x))dx  = \int \cos(s v)d\cos v,\qquad v:=\arccos x\\
& = -\int \cos(sv)\sin v dv \\
& = -\frac{1}{2}\int [\sin((s+1)v) - \sin((s-1)v)] dv\\
& = \frac{1}{2}\bigg[\frac{\cos((s+1)v)}{s+1} - \frac{\cos((s-1)v)}{s-1}\bigg] + C\\
& = \frac{1}{2}\bigg[\frac{T_{s+1}(x)}{s+1} - \frac{T_{s-1}(x)}{s-1}\bigg]+C.
\end{split}
\end{equation}
Lemma \ref{lemma1} is proved. 
\end{proof}

From \eqref{eq31} and \eqref{eq23} we need to find $(c_i)_{i=0}^s$ so that
\begin{equation}
\label{eq25}
\prod_{i=0}^s (x-c_i) =\frac{s+1}{2^{s-1}} 
\bigg[ \frac{1}{2}\bigg(\frac{T_{s+1}(x)}{s+1} - \frac{T_{s-1}(x)}{s-1}\bigg) + C\bigg],
\end{equation}
where $C$ is a constant. 
However, it is not clear if there is a constant $C$ so that there exists $(c_i)_{i=0}^s$, $c_i\in [-1,1]$, satisfying equation \eqref{eq25}. 

Consider the case when $s$ is odd. 
Let
\begin{equation}
\label{eq26}
P_{s+1}(x):= \frac{s+1}{2^{s-1}} 
\bigg[ \frac{1}{2}\bigg(\frac{T_{s+1}(x)}{s+1} - \frac{T_{s-1}(x)}{s-1}\bigg) + \frac{1}{s^2-1}\bigg].
\end{equation}
We have the following result:
\begin{theorem}
\label{thm3.2xx}
Let $s>0$ be an odd integer.  
The polynomial $P_{s+1}(x)$ defined in \eqref{eq26} has $s+1$ distinct zeros $(c_i)_{i=0}^s$ on the interval $[-1,1]$, $-1=c_0<c_1<...<c_s=1$. These zeros are symmetric about 0. 
\end{theorem}

\begin{proof}
When $s$ is odd, the polynomials $T_{s+1}(x)$ and $T_{s-1}(x)$ are even functions on $[-1,1]$. Thus, $P_{s+1}(x)$ is an even function and its zeros are symmetric about 0. 

Since $s+1$ and $s-1$ are even, one has $T_{s+1}(\pm 1) = T_{s-1}(\pm 1) = 1$. Thus, 
\begin{equation}
\label{eq36}
P_{s+1}(1) = P_{s+1}(-1) = 0.
\end{equation}

From \eqref{eq23} and \eqref{eq26} one gets $P'_{s+1}(x) = \frac{s+1}{2^{s-1}}T_s(x)$. 
Thus, $P'_{s+1}(x)$ and $T_s(x)$ share the same zeros which are $(x_i)_{i=0}^{s-1}$, $x_i=\cos(\frac{2i + 1}{2s})$. 
Since $T_{s+1}(x) + T_{s-1}(x) = 2xT_s(x)$ and $(x_i)_{i=0}^{s-1}$ are zeros of $T_s(x)$, one gets $T_{s+1}(x_i) + T_{s-1}(x_i) = 2x_iT_s(x_i)=0$. Thus, $T_{s+1}(x_i) = - T_{s-1}(x_i)$, $i=0,...,s-1$, and from \eqref{eq26} one gets
\begin{equation}
\label{eq33}
\begin{split}
P_{s+1}(x_i) &= \frac{s+1}{2^{s-1}} 
\bigg[ \frac{1}{2}\bigg(\frac{T_{s+1}(x_i)}{s+1} - \frac{T_{s-1}(x_i)}{s-1}\bigg) + \frac{1}{s^2-1}\bigg]\\
&= \frac{s+1}{2^{s-1}} 
\bigg[ \frac{T_{s+1}(x_i)}{2}\bigg(\frac{1}{s+1} + \frac{1}{s-1}\bigg) + \frac{1}{s^2-1}\bigg]\\
& = \frac{s+1}{2^{s-1}(s^2 - 1)} 
(sT_{s+1}(x_i)+1),\qquad i=0,...,s-1.
\end{split}
\end{equation}
From the relation $\arccos(x_i) = \frac{2i+1}{2s}\pi$, one gets
\begin{equation}
\label{eq34}
\begin{split}
T_{s+1}(x_i) &= \cos((s+1)\arccos(x_i)) \\
&= \cos\bigg(\frac{2i+1}{2}\pi + \frac{2i+1}{2s}\pi\bigg) = (-1)^{i+1}\sin\bigg(\frac{2i+1}{2s}\pi\bigg). 
\end{split}
\end{equation}
Note that
$$
x - \sin(\frac{\pi x}{2}) < 0,\quad \forall x\in(0,1).
$$
Thus,
\begin{equation}
\label{eq35}
\sin\bigg(\frac{2i+1}{2s}\pi\bigg) \ge \sin\bigg(\frac{\pi}{2s}\bigg) > \frac{1}{s},\qquad i=0,...,s-1,\quad s>1. 
\end{equation}
From \eqref{eq33}--\eqref{eq35} one obtains
\begin{equation}
\label{eq37}
\begin{split}
(-1)^{i+1} P_{s+1}(x_i) &= (-1)^{i+1} \frac{s}{(s-1)2^{s-1}}\bigg(T_{s+1}(x_i) + \frac{1}{s}\bigg)\\
&= \frac{s}{(s-1)2^{s-1}}\bigg( \sin\bigg(\frac{2i+1}{2s}\pi\bigg) + \frac{(-1)^{i+1}}{s}\bigg)\\
&> \frac{s}{(s-1)2^{s-1}}\bigg(\frac{1}{s} + \frac{(-1)^{i+1}}{s}\bigg) \ge 0,\qquad i=0,...,s-1.
\end{split}
\end{equation}
Thus, $P_{s+1}(x)$ has at least $s-1$ zeros on the interval $[x_0,x_{s-1}]\subset (-1,1)$. 
This and \eqref{eq36} imply that $P_{s+1}(x)$ has $s+1$ zeros on the interval $[-1,1]$.

Theorem \ref{thm3.2xx} is proved. 
\end{proof}

Consider the case when $s$ is an even integer. 
\begin{theorem}
\label{theorem3}
Let $0<s$ be an even integer. 
For any constant $C$ the polynomial 
\begin{equation}
\label{eq25.x}
g_{s+1}(x) = \frac{s+1}{2^{s-1}} 
\bigg[ \frac{1}{2}\bigg(\frac{T_{s+1}(x)}{s+1} - \frac{T_{s-1}(x)}{s-1}\bigg) + C\bigg],
\end{equation} 
has at most $s$ zeros on the interval $[-1,1]$. 
\end{theorem}

\begin{proof}
From Lemma \ref{lemma1} and \eqref{eq25.x} one gets $g'_{s+1}(x) = \frac{s+1}{2^{s-1}}T_s(x)$. Thus, zeros of $g'_{s+1}(x)$ are zeros of $T_s(x)$ which are $(x_i)_{i=0}^{s-1}$, $x_i = \cos(\frac{2i+1}{2s}\pi)$, $i=0,...,s-1$.

Since $g'_{s+1}(x)$ does not change sign on intervals $[x_i,x_{i+1}]$, $i=0,...,s-2$, the function $g_{s+1}(x)$ 
has at most $s-1$ zeros on $[x_0,x_{s-1}]$ for any given $C$. 

One has
\begin{equation}
g'_{s+1}(x) = \frac{s+1}{2^{s-1}}T_s(x) \ge 0,\qquad -x\in [-1,x_0]\cup [x_{s-1},1]. 
\end{equation}
Thus
\begin{equation}
\min_{[-1,x_0]} g_{s+1}(x) \ge \frac{1}{s^2-1} + C >  -\frac{1}{s^2-1} + C\ge \max_{x\in [x_{s-1},1]}g_{s+1}(x).
\end{equation}
Thus, for any given $C$ there exists at most one zero of $g_{s+1}(x)$ on $[-1,x_0]\cup [x_{s-1},1]$. 
Therefore, the function $g_{s+1}(x)$ has at most $s$ zeros on the interval $[-1,1]$. 
\end{proof}

\begin{remark}
Theorem \ref{theorem3} says that 
when $s$ is even, there does not exist a constant $C$ and $(c_i)_{i=0}^s$, $c_i\in [-1,1]$, so that equation  \eqref{eq25} holds. Thus, when $s$ is even, there does not exist a node distribution $(c_i)_{i=0}^s$, $-1=c_0<c_1<...<c_s=1$, so that equation \eqref{eq22} holds. 
\end{remark}

Let us propose two possible node distributions $(c_i)_{i=0}^s$ for computing differentiation matrices when $s$ is even. 
Let
\begin{equation}
\label{eq41}
\tilde{Q}_{s+1}(x) := \frac{s+1}{2^{s}} 
\bigg(\frac{T_{s+1}(x)}{s+1} - \frac{T_{s-1}(x)}{s-1} + \frac{2x}{s^2-1}\bigg).
\end{equation}
Then $\tilde{Q}_{s+1}(x)$ has $s+1$ zeros $(c_i)_{i=0}^s$ on the interval $[-1,1]$, $-1=c_0<c_1<...<c_s=1$. 
Thus, one can use this node distribution for the computation of differential matrices. 

\begin{theorem}
\label{thm3.4}
Let $s$ be an even integer. 
The polynomial $\tilde{Q}_{s+1}(x)$ has $s+1$ zeros $(c_i)_{i=0}^{s}$ on the interval $[-1,1]$, $-1=c_0<c_1<...<c_s = 1$. 
\end{theorem}

\begin{proof}
When $s$ is even, we have $T_{s+1}(\pm 1)=\pm 1$, and $T_{s-1}(\pm 1)=\pm 1$. 
Thus, one gets
\begin{equation}
\label{eq38}
\tilde{Q}_s(-1) = \tilde{Q}_s(1) = 0. 
\end{equation}

Let $x_i=\cos(\frac{2i+1}{2s}\pi)$, $i=0,...,s-1$. Then $(x_i)_{i=0}^{s-1}$ are zeros of $T_{s}(x)$. 
By similar arguments as in Theorem \ref{thm3.2xx} (cf. \eqref{eq37}) one gets
\begin{equation}
\begin{split}
(-1)^{i+1} \tilde{Q}_{s+1}(x_i) &= (-1)^{i+1} \frac{s}{(s-1)2^{s-1}}\bigg(T_{s+1}(x_i) + \frac{x_i}{s}\bigg)\\
&= \frac{s}{(s-1)2^{s-1}}\bigg( \sin\bigg(\frac{2i+1}{2s}\pi\bigg) + \frac{(-1)^{i+1}x_i}{s}\bigg)\\ 
&> \frac{s}{(s-1)2^{s-1}}\bigg( \frac{1}{s} - \frac{|x_i|}{s}\bigg) > 0,\qquad i=0,...,s-1.
\end{split}
\end{equation}
Thus, $\tilde{Q}_{s+1}(x)$ has at least $s-1$ zeros on the interval $[x_0,x_{s-1}]$. Taking into account \eqref{eq38}, one concludes that $\tilde{Q}_{s+1}(x)$ has $s+1$ zeros on the interval $[-1,1]$.

Theorem \ref{thm3.4} is proved. 
\end{proof}

If $(c_i)_{i=0}^s$ are chosen as zeros of $\tilde{Q}_{s+1}(x)$, then we have
\begin{equation}
\frac{d}{dx} \prod_{i=0}^s(x-c_i) = \frac{s+1}{2^{s-1}}\bigg(T_s(x) + \frac{1}{(s-1)(s+1)}\bigg). 
\end{equation}
Thus, for this choice of $(c_i)_{i=0}^s$ equation \eqref{eq22} is not satisfied. 

Let us discuss another possible choice for $(c_i)_{i=0}^s$ when $s$ is even. 
Consider the following polynomial
\begin{equation}
Q_{s+1}(x) := \frac{s+1}{2^{s}} 
\bigg(\frac{T_{s+1}(x)}{s+1} - \frac{T_{s-1}(x)}{s-1}\bigg).
\end{equation}
The functions $T_{s-1}(x)$ and $T_{s+1}(x)$ are odd functions when $s$ is an odd integer. 
Thus, $Q_{s+1}(s)$ is an odd function when $s$ is even. 
Therefore, zeros of $Q_{s+1}(s)$ are symmetric about $0$.  
By similar arguments as in Theorem ~\ref{thm3.2xx} one can show that $Q_{s+1}(s)$ has $s+1$ zeros. Note that not all these $s+1$ zeros are in $[-1,1]$.  
Let $d_0<d_1<...<d_s$ be zeros of $Q_{s+1}(s)$ and let $c_i := \frac{d_i}{d_s}$. Then
\begin{equation}
\label{eq44}
-1=c_0<c_1<...<c_s = 1,\quad c_i = \frac{d_i}{d_s},\qquad i=0,...,s. 
\end{equation}
Note that if $(c_i)_{i=0}^s$ are chosen by \eqref{eq44}, then equation \eqref{eq22} does not hold. 
In fact, for this choice of $(c_i)_{i=0}^s$ one has
\begin{equation}
\frac{d}{dx} \prod_{i=0}^s (x-c_i) = \bigg(\frac{1}{d_s}\bigg)^{s+1} \frac{s+1}{2^{s-1}} T_{s}(d_s x). 
\end{equation}

\section{Numerical experiments}
\label{sec4}

\subsection{Interpolation}

In this section we will carry out numerical experiments to compare the Lebesgue constants of the following node distributions: 

1. Chebyshev-Gauss-Lobatto points
\begin{equation}
c_i = \cos(\frac{i}{s}\pi),\qquad i=0,...,s.
\end{equation}

2. Scaled Chebyshev points
\begin{equation}
c_i = \frac{\cos(\frac{2i+1}{2s+2}\pi)}{\cos(\frac{1}{2s+2}\pi)},\qquad i=0,...,s.
\end{equation}

3. Equidistant nodes
\begin{equation}
c_i = - 1 + \frac{2i}{s},\qquad i=0,...,s.
\end{equation}

The Lebesgue constant $\Lambda(\bm{c})$ can be computed by the formula (see, e.g., \cite{FB})
\begin{equation}
1 + \Lambda(\bm{c}) = \max_{x\in [-1,1]} F_{\bm{c}}(x),\qquad 
F_{\bm{c}}(x):= \sum_{k=0}^s |F_{k}(x)|,
\end{equation}
where 
\begin{equation}
F_{k}(x) = \prod_{j=0 \atop{j\not=k}}^s (x-c_j)\bigg/ \prod_{j=0\atop{j\not=k}}^s (c_k-c_j),\qquad k=0,...,s.
\end{equation}

In all experiments, we denote by CGL the numerical solutions obtained by using Chebyshev-Gauss-Lobatto node distribution. 

Figure \ref{fig1} plots the function $F_{\bm{c}}(x)$ based on equidistant nodes, Chebyshev-Gauss-Lobatto nodes, and the extended Chebyshev nodes studied in this paper. From Figure \ref{fig1} one can see that 
the scaled Chebyshev node distribution yields a function $F_{\bm{c}}(x)$ with minimal sup-norm among the three node distributions.  

\begin{figure}[!h!t!b!]
\centerline{
\begin{tabular}{c}ets
\mbox{\includegraphics[scale=0.74]{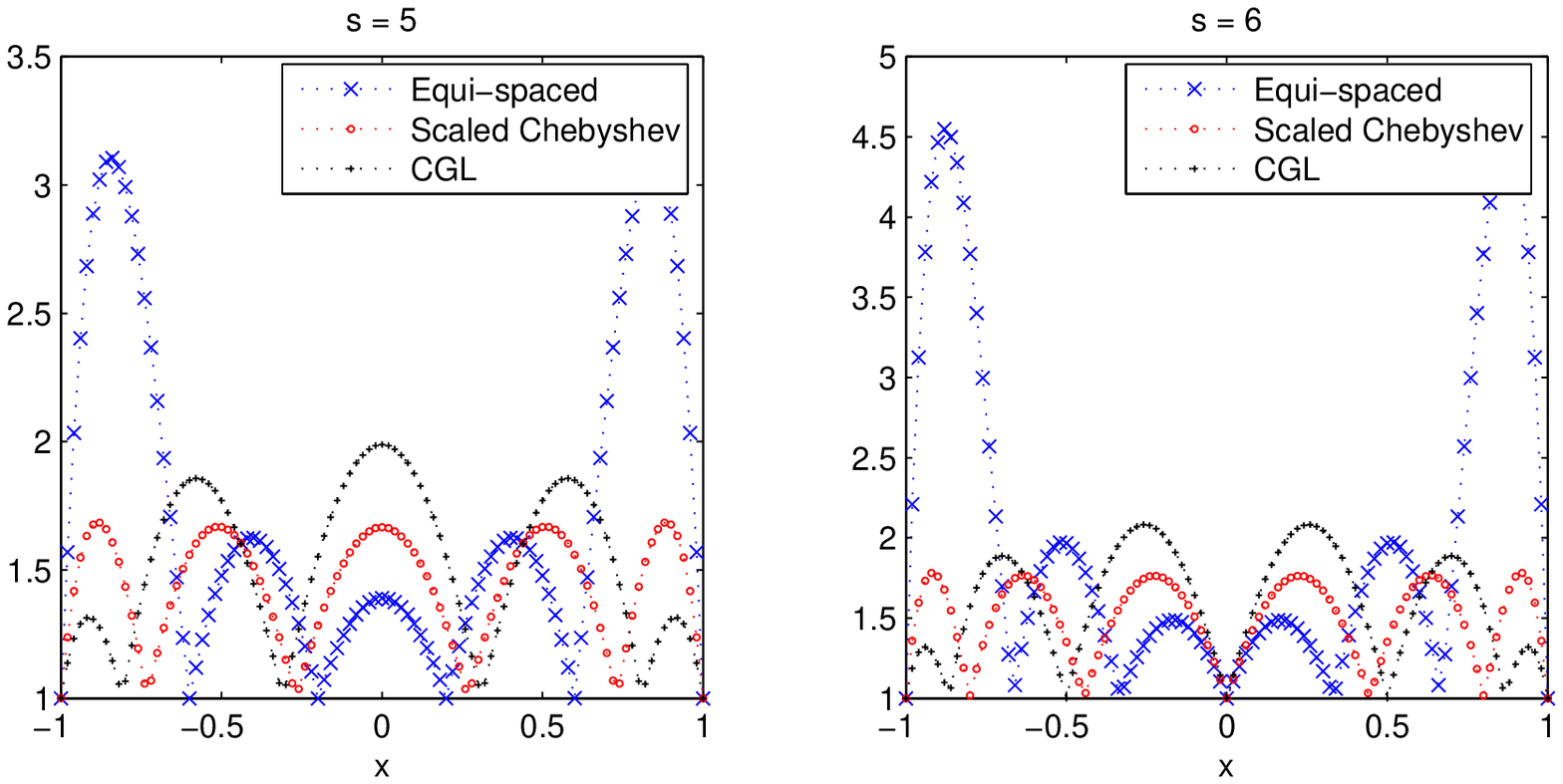}}
\end{tabular}
}
\caption{\it Plots of $F_{\bm{c}}(x)$ when $s=4$ and $s=6$. 
}
\label{fig1}
\end{figure}

Table \ref{table1} below presents Lebesgue constants for the three node distributions for various $s$. 
From Table \ref{table1} one concludes that the scaled Chebyshev nodes yield the smallest Lebesgue constants among the three node distributions. One can also see that the Lebesgue constant $\Lambda(n)$ of equidistant node distribution increases very fast when $n$ increases. 

\begin{table}[ht]
\centering
\small
\begin{tabular}{|c@{\hspace{2mm}}
@{\hspace{2mm}}|c@{\hspace{2mm}}c@{\hspace{2mm}}c@{\hspace{2mm}}c@{\hspace{2mm}}c@{\hspace{2mm}}c@{\hspace{2mm}}c@{\hspace{2mm}}|}
\hline 
      Node distribution & $s= 6$	& $s= 8$	&$s= 10$	&$n=12$	   &$n=14$   &$n=16$ &$n=18$\\
\hline
Equi-spaced             &3.6   &9.9   &28.9   &88.3  &282.2 &933.5 &3170.1\\
\hline
Chebyshev-Gauss-Lobatto &1.1   &1.3   &1.4    &1.5   &1.6   &1.7   &1.8\\
\hline
Scaled Chebyshev        &0.8   &0.9   &1.1    &1.2   &1.3   &1.3   &1.4\\
\hline
\end{tabular}
\caption{Lebesgue constants}
\label{table1}
\end{table}

Figure \ref{fig2} plots absolute values of errors of Lagrange interpolation $|f(x) - L_{\bm{c}}(f)(x)|$ for  equidistant nodes, Chebyshev-Gauss-Lobatto nodes, and the scaled Chebyshev nodes when $f(x)=e^{x}$ (left) and $f(x) = \cos(x)$ (right). From Figure \ref{fig2} one concludes that the scaled Chebyshev node distribution is the best among the three node distributions in this experiment. 

\begin{figure}[!h!t!b!]
\centerline{
\begin{tabular}{c}
\mbox{\includegraphics[scale=0.74]{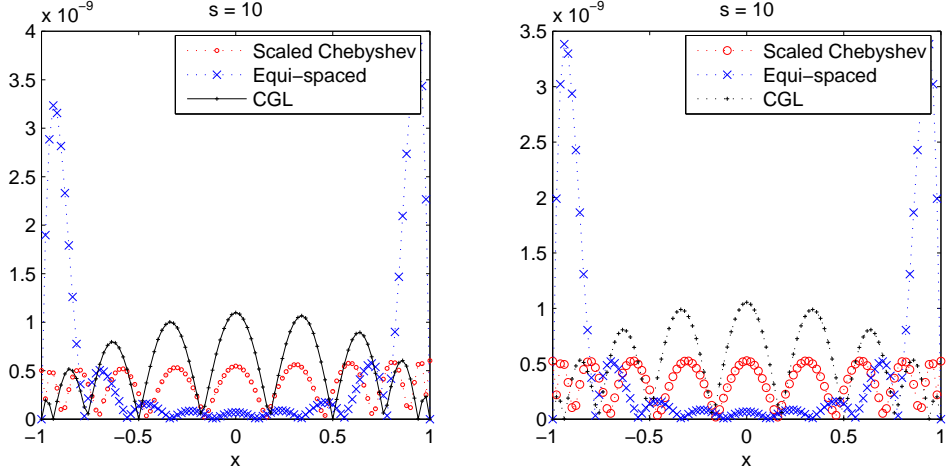}}
\end{tabular}
}
\caption{\it Plots of interpolating errors for $f(x)=e^{x}$ (left) and $f(x) = \cos(x)$ (right). 
}
\label{fig2}
\end{figure}

\subsection{Numerical differentiation}
\label{}

The Lagrange interpolation polynomial $L_{\bm{c}}(f)$ of $f$ over the nodes $(c_i)_{i=0}^s$ is given by
\begin{equation}
\label{eq39}
L_{\bm{c}}(f)(x) = \sum_{i=0}^s f(c_i)\ell_{\bm{c},i}(x),\quad \ell_{\bm{c},i}(x) = \prod_{j=0\atop{j\not=i}}^s (x-c_j)/ \prod_{j=0\atop{j\not=i}}^s (c_i-c_j).
\end{equation}
Therefore,
\begin{equation}
(L_{\bm{c}}(f))'(x) = \sum_{i=0}^s f(c_i)\ell'_{\bm{c},i}(x).
\end{equation}
This implies
\begin{equation}
(L_{\bm{c}}(f))'(c_k) = \sum_{i=0}^s f(c_i)\ell'_{\bm{c},i}(c_k),\qquad k=0,...,s.
\end{equation}
These equations can be rewritten as
\begin{equation}
\label{au01e1}
\begin{pmatrix}
(L_{\bm{c}}(f))'(c_0)\\
(L_{\bm{c}}(f))'(c_1)\\
\vdots \\
(L_{\bm{c}}(f))'(c_s)
\end{pmatrix}
 = 
\begin{pmatrix}
d_{00}  &d_{01}  &\cdots  &d_{0s}\\
d_{10}  &d_{11}  &\cdots  &d_{1s}\\
\vdots  &\vdots  &\ddots  &\vdots\\
d_{s0}  &d_{s1}  &\cdots  &d_{ss}
\end{pmatrix}
\begin{pmatrix}
f(c_0)\\
f(c_1)\\
\vdots\\
f(c_s)
\end{pmatrix},
\qquad d_{ij} := \ell'_{\bm{c},j}(c_i).
\end{equation}
The matrix $D=(d_{ij})_{i,j=0}^s$ is called a differentiation matrix. 
The derivatives $f'(c_i)$, $i=0,...,s$, are approximated by $(L_{\bm{c}}(f))'(c_i)$ which are computed by \eqref{au01e1}. 

Let us derive formulae for computing the differentiation matrix $D=(d_{ij})_{i,j=0}^s$. 
From \eqref{eq39}, one gets
\begin{equation}
\ell'_{\bm{c},i}(x) = \ell_{\bm{c},i}(x) \sum_{j=0\atop{j\not=i}}^s \frac{1}{x-c_j},\qquad x\not = c_j.
\end{equation}
Thus,
\begin{gather}
d_{ji}=\ell'_{\bm{c},i}(c_j) = \prod_{k=0\atop{k\not=i,j}}^s (c_i-c_k)\bigg/ \prod_{k=0\atop{k\not=j}}^s (c_j-c_k),\quad i\not=j,\quad i,j=0,...,s,\\
d_{ii}=\ell'_{\bm{c},i}(c_i) = \sum_{k=0\atop{k\not=i}}^s (c_i-c_k),\qquad i=0,...,s.
\end{gather}
One can find similar formulae in \cite{Trefethen}. 


When $(c_i)_{i=0}^s$ are Chebyshev-Gauss-Lobatto points, the differentiation matrix $D = (d_{ij})_{i,j=0}^s$ is given by (see, e.g., \cite{Trefethen})
\begin{gather}
d_{00}= \frac{2s^2 + 1}{6},\qquad d_{ss} = \frac{2s^2 + 1}{6},\\
d_{jj} = \frac{-c_j}{2(1-c_j^2)},\qquad j=1,...,s-1,\\
d_{ij} = \frac{a_i}{a_j} \frac{(-1)^{i+j}}{c_i - c_j},\qquad i\not=j,\quad i,j=0,...,s,
\end{gather}
where
\begin{equation}
a_i = \bigg \{\begin{matrix} 2,& i = 0\quad \text{or}\quad i=s,\\
1,& \text{otherwise}.
\end{matrix}
\end{equation}

Let us do some numerical experiments with the computation of the first  derivative of a function $f$ using different types of node distributions. These node distributions are Chebyshev-Gauss-Lobatto points, equi-spaced distribution, the scaled Chebyshev points, and the node distribution developed in Section \ref{sec3}. In our experiments, the node distribution from Theorem \ref{thm3.2xx} is denoted by ND1 and the node distribution from Theorem \ref{thm3.4} is denoted by ND2. 

Figure \ref{figure3} plots the errors $|f'(c_i)- (L_{\bm{c}}(f))'(c_i)|$ for the four node distributions for the function $f(x)=e^{x}$. 
From Figure \ref{figure3} one can see that the node distribution ND1 studied in this paper yields the best results in the sup-norm. The approximation for $(f'(c_i))_{i=0}^s$ with equidistant nodes are very good when $c_i$ is close to $0$ but are not good when $c_i$ is close to the boundary $-1$ or $1$. The accuracy of numerical solutions from all node distributions in this experiment is high even with ten nodes.

\begin{figure}[!h!t!b!]
\centerline{
\begin{tabular}{c}
\mbox{\includegraphics[scale=0.74]{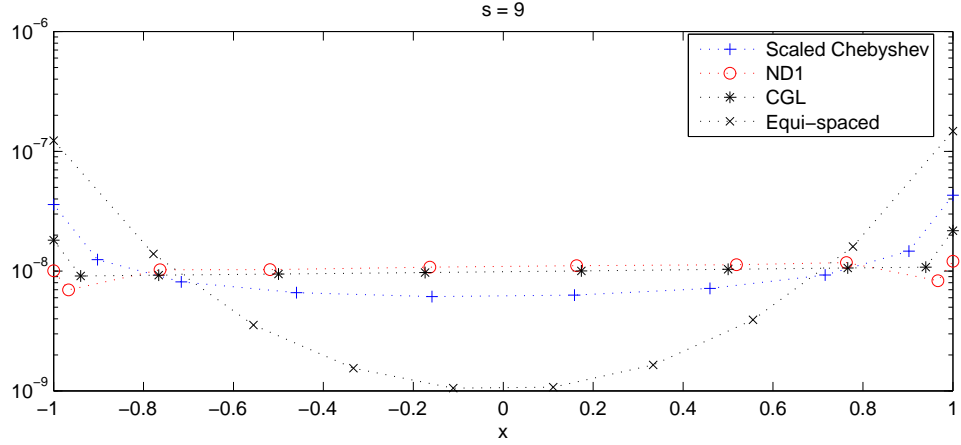}}
\end{tabular}
}
\caption{\it Plots of $|f'(c_i)- (L_{\bm{c}}(f))'(c_i)|$, $i=0,...,s$, for $f(x)=e^{x}$.
}
\label{figure3}
\end{figure}

Figure \ref{figure4} plots the errors $|f'(c_i)- (L_{\bm{c}}(f))'(c_i)|$ for the four node distributions for the function $f(x)=e^{x^2}$. 
From Figure \ref{figure4} one can see that the result obtained from the node distribution ND1 is the best in the sup-norm. 
Again, the numerical approximations to $f'(c_i)$, $i=0,...,s$, with equidistant nodes are very good when $c_i$ is close to $0$ but are not good when $c_i$ is close to the boundary $-1$ or $1$. 
The accuracy of numerical solutions in this experiment is not very high since the function $e^{x^2}$ in this experiment grows much faster than the function $e^x$ in the previous experiment.

\begin{figure}[!h!t!b!]
\centerline{
\begin{tabular}{c}
\mbox{\includegraphics[scale=0.74]{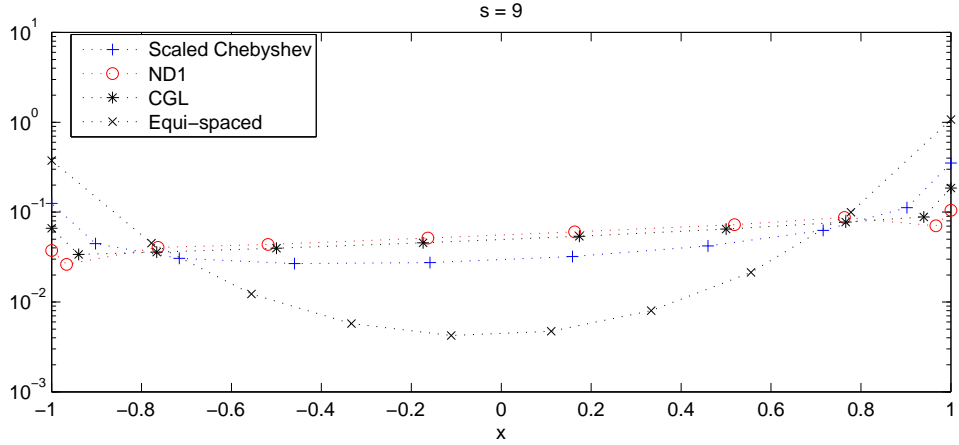}}
\end{tabular}
}
\caption{\it Plots of $|f'(c_i)- (L_{\bm{c}}(f))'(c_i)|$, $i=0,...,s$, for $f(x)=e^{x^2}$.
}
\label{figure4}
\end{figure}

Figures \ref{figure5} and \ref{figure6} plot numerical results for the four node distributions: Chebyshev-Gauss-Lobatto node distribution, the scaled Chebyshev node distribution, the equi-spaced nodes, and the node distribution ND2.

Figure \ref{figure5} plots the numerical errors for computing $f'(c_i)$, $i=0,...,s$, for $f(x) = e^x$, on $[-1,1]$. 
It is clear from Figure \ref{figure5} that the ND2 node distribution yields the best result and the equi-spaced node distribution yields the worst result. From Figure \ref{figure5} we conclude that Chebyshev-Gauss-Lobatto nodes work better than the scaled Chebyshev nodes in this experiment. 

\begin{figure}[!h!t!b!]
\centerline{
\begin{tabular}{c}
\mbox{\includegraphics[scale=0.74]{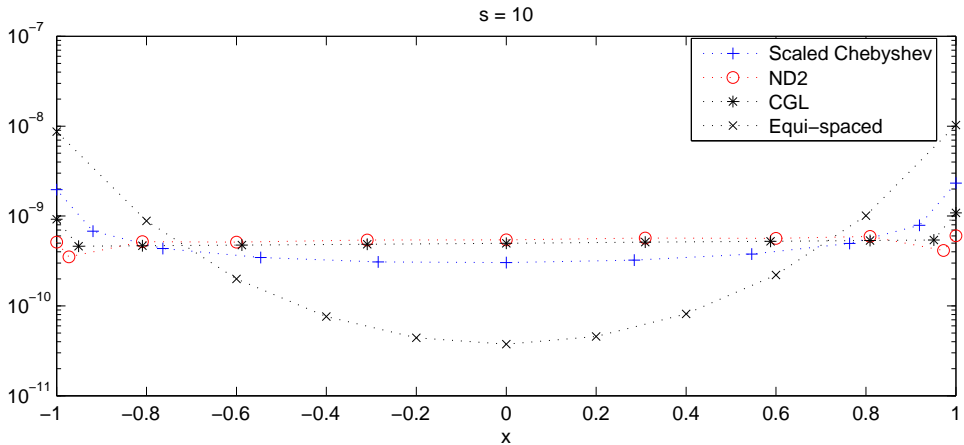}}
\end{tabular}
}
\caption{\it Plots of $|f'(c_i)- (L_{\bm{c}}(f))'(c_i)|$, $i=0,...,s$, for $f(x)=e^{x}$.
}
\label{figure5}
\end{figure}

Figure \ref{figure6} plots the results for $f=e^{x^2}$. Again, it follows from Figure \ref{figure6} that 
the ND2 yields the best numerical result. It is clear from Figure \ref{figure6} that Chebyshev-Gauss-Lobatto nodes work better than the scaled Chebyshev nodes. The equi-spaced node distribution is the worst among these node distributions. 

\begin{figure}[!h!t!b!]
\centerline{
\begin{tabular}{c}
\mbox{\includegraphics[scale=0.74]{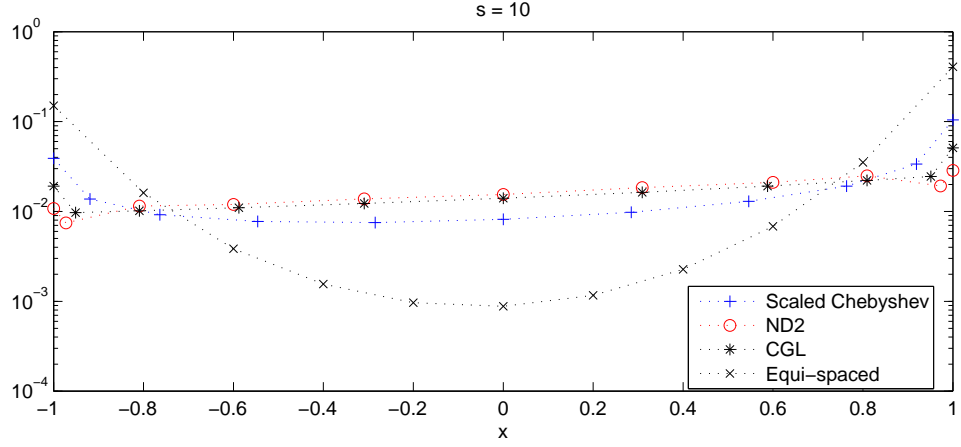}}
\end{tabular}
}
\caption{\it Plots of $|f'(c_i)- (L_{\bm{c}}(f))'(c_i)|$, $i=0,...,s$, for $f(x)=e^{x^2}$.
}
\label{figure6}
\end{figure}

\subsection{Solving a Volterra equation of the first kind}

Let us do a numerical experiment with solving the following equation
\begin{equation}
\label{eqj30}
\int_0^t K(t,\xi)u(\xi) d\xi = f(t),\qquad 0\le t\le 1.
\end{equation}
To solve equation \eqref{eqj30} we approximate $u(\xi)$ by its Lagrange interpolation polynomial $P(\xi)$ over the nodes $(c_i)_{i=0}^s$ and solve for $(u(c_i))_{i=0}^s$ from equation \eqref{eqj30}. 
In particular, we have
\begin{equation}
u(\xi) \approx P(\xi)=\sum_{j=0}^s \ell_j(\xi)u(c_j),\qquad \xi \in [0,1].
\end{equation}
From equation \eqref{eqj30} one gets
\begin{equation}
\label{eqjl25}
\begin{split}
f(c_i) &\approx \int_0^{c_i} K(c_i,\xi)P(\xi)d\xi
 = \sum_{j=0}^s u(c_j) \int_0^{c_i}\ell_j(\xi)K(c_i,\xi) d\xi,\quad i=0,...,s.
\end{split}
\end{equation}
Equation \eqref{eqjl25} can be written as
\begin{equation}
\label{au6e1}
A_su_s \approx f_s,
\end{equation}
where
\begin{gather}
u_s = (u(c_0), u(c_1),...,u(c_s))^T,\quad f_s = (f(c_0), f(c_1),...,f(c_s))^T,\\
A_s = (a_{ij})_{i,j=0}^s,\qquad a_{ij} = \int_0^{c_i} \ell_j(\xi) K(c_i,\xi)d\xi\label{eqjl26.1}.
\end{gather}
Taking into account \eqref{au6e1}, we solve for $\tilde{u}_s$ from the linear algebraic system 
\begin{equation}
A_s\tilde{u}_s = f_s,
\end{equation}
and take $\tilde{u}_s$ as an approximation to $u_s=(u(c_0), u(c_1),...,u(c_s))^T$. 

In our experiments we choose $K(t,\xi) = e^{t-\xi}$ and $u(t)=\cos(\pi t),\, t\in [0,1]$. 
We compare the three distributions: Chebyshev-Gauss-Lobatto points, the scaled Chebyshev points, and the two node distributions developed in Section \ref{sec3}. The elements $a_{ij}$, $i,j=0,...,s$, in 
equation \eqref{eqjl26.1} are computed by means of quadrature formulas. In fact, we used the function {\it quad} in MATLAB to compute these coefficients. 

Figure \ref{figure7} plots the results obtained by using Chebyshev-Gauss-Lobatto points, the scaled Chebyshev points, and the node distribution ND1 developed in Section \ref{sec3} for the case when $s=9$. 
From Figure \ref{figure7} we can see that the node distribution ND1 yields the best result in the sup-norm. 
The scaled Chebyshev node distribution yields the worst result in sup-norm in this experiment. 
\begin{figure}[!h!t!b!]
\centerline{
\begin{tabular}{c}
\mbox{\includegraphics[scale=0.68]{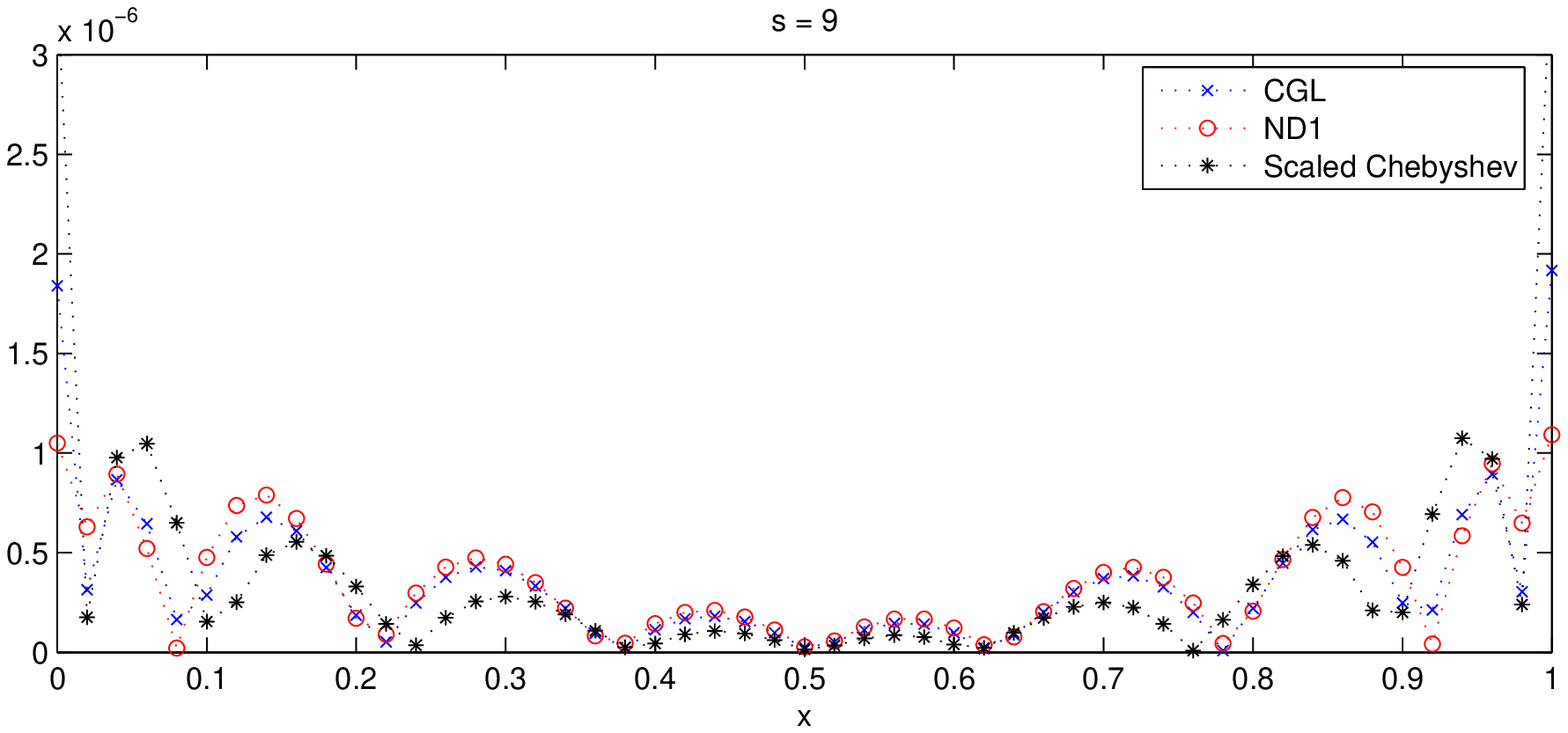}}
\end{tabular}
}
\caption{\it Plots of absolute values of errors for $u(\xi)=\cos(\pi \xi)$.
}
\label{figure7}
\end{figure}

Figure \ref{figure8} plots the results obtained by using Chebyshev-Gauss-Lobatto points, the scaled Chebyshev points, and the node distribution ND2 developed in Section \ref{sec3} for the case when $s=10$.  
We can see from Figure \ref{figure8} that the result obtained by using the node distribution ND2 is the best in the sup-norm. Again, the result obtained by using the scaled Chebyshev node distribution is the worst in sup-norm in this experiment.
\begin{figure}[!h!t!b!]
\centerline{
\begin{tabular}{c}
\mbox{\includegraphics[scale=0.68]{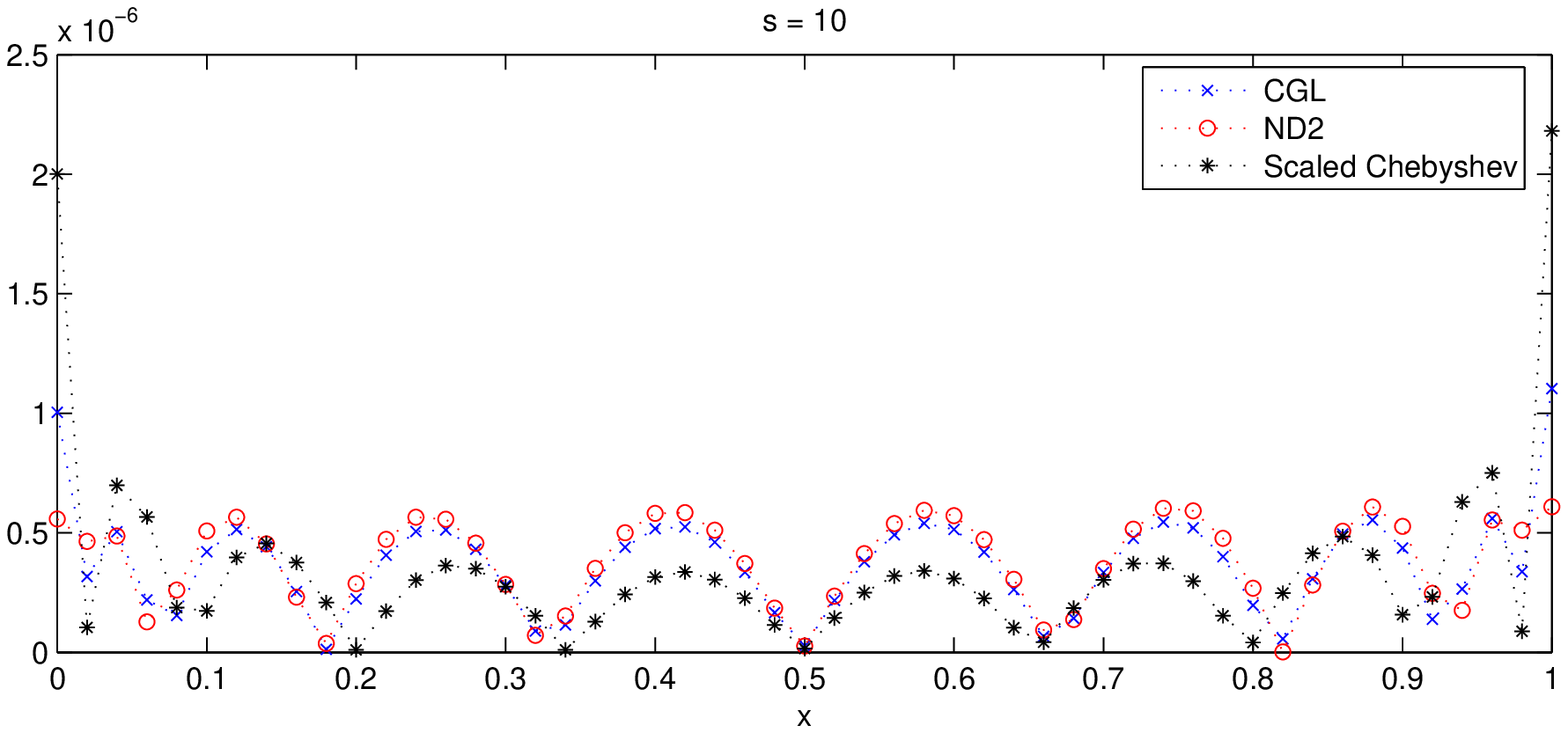}}
\end{tabular}
}
\caption{\it Plots of absolute values of errors for $u(\xi)=\cos(\pi \xi)$.
}
\label{figure8}
\end{figure}

\end{document}